\documentclass[11pt]{amsart}
\usepackage{amsthm} 
\usepackage{color}
\usepackage{graphicx, epsfig}
\usepackage{epstopdf}
\usepackage{amsmath, amssymb, latexsym, euscript, amscd}
\usepackage{url}
\usepackage[all]{xy}
\usepackage{psfrag}
\usepackage{mathrsfs}
\usepackage{caption, wrapfig, multirow, tabularx, mathrsfs,verbatim}
\usepackage{subcaption}
\usepackage{mathptmx}
\usepackage[]{overpic}
\usepackage[dvipsnames]{xcolor}
\usepackage{thmtools}
\usepackage[colorlinks=true, linkcolor=blue, citecolor = blue, urlcolor=blue]{hyperref}
\usepackage[]{cleveref}
\usepackage{tikz-cd} 
\usepackage[margin=1.50in]{geometry}

\usepackage{xcolor}

\definecolor{brightPurple}{HTML}{444FAD}
\definecolor{brightTeal}{HTML}{19B092}
\definecolor{brightPink}{HTML}{FF0082}
\definecolor{brightIndigo}{HTML}{667BFB}
\definecolor{brightBlue}{HTML}{3498DB}
\definecolor{brightGrey}{HTML}{9DBBD8}

\newtheorem{theorem}{Theorem}[section]
\newtheorem{lemma}[theorem]{Lemma}
 
\newtheorem{definition}[theorem]{Definition} 
\newtheorem{proposition}[theorem]{Proposition}

\theoremstyle{remark}
\newtheorem{remark}[theorem]{Remark}

\def\gap{\vspace{.3cm}\noindent}

\def\smallskip{\vspace{.15cm}}
\def\medskip{\vspace{.3cm}}

\def\Ucal{\mathcal U}

\def\Wcal{\mathcal W}

\def\Vcal{\mathcal V}

\def\bdy{\partial}
\def\RR{\mathbb R}
\def\ZZ{\mathbb Z}
\def\ZZ{\mathbb Z}

\def\lab{\mathbb L}
\def\EuM{\EuScript{M}}
\def\EuT{\EuScript{T}}
\def\EuW{\EuScript{W}}
\def\EuU{\EuScript{U}}
\def\EuV{\EuScript{V}}

\def\cl{\operatorname{cl}}

\def\bdy{\partial}
\def\bd{\partial}

\def\GL{\operatorname{GL}}
\def\star{\operatorname{star}}

\def\intr{\textnormal{int}}
\def\TR{\textnormal{Tr}}
\def\PL{\textnormal{PL}}

\newcommand{\bv}{\left[\begin{array}{c}}
\newcommand{\ev}{\end{array}\right]}
\newcommand{\bbmat}{\begin{bmatrix}} 
\newcommand{\ebmat}{\end{bmatrix}}
\newcommand{\bmat}{\begin{matrix}} 
\newcommand{\emat}{\end{matrix}}
\newcommand{\bpmat}{\begin{pmatrix}} 
\newcommand{\epmat}{\end{pmatrix}}
\newcommand{\icol}[1]

\title{Combinatorial Characterizations and Branched Manifolds}

\author{Daryl Cooper}
\address{Mougins, France}
\email{cooper@math.ucsb.edu}

\author{Leslie Mavrakis}
\address{Department of Mathematics\\University of Utah\\Salt Lake City, UT 84112}
\email{l.mavrakis@utah.edu}

\author{Priyam Patel}
\address{Department of Mathematics\\University of Utah\\Salt Lake City, UT 84112}
\email{patelp@math.utah.edu}

\date{}

\begin{document}

\begin{abstract}
A family of compact $n$-manifolds is {\em locally combinatorially defined} (LCD) if it can be specified by a finite number of local triangulations. We show that LCD is equivalent to the existence of a compact branched $n$-manifold $W$, such that the family is precisely those manifolds that immerse into $W$. In subsequent papers, the equivalence will be used to show that, for each of the eight Thurston geometries, the family of closed 3-manifolds admitting that geometry is LCD.
\end{abstract}

\maketitle
\section{Introduction} 

Many interesting families of geometric objects are specified by their local structure. For instance, a {\em locally homogeneous space} is one that is locally modeled on $G/H$, where $G$ is a Lie group and $H$ is a closed subgroup. Classical examples include Riemannian manifolds of constant curvature, as well as real projective manifolds. Our interest here is in exploring a combinatorial analogue of this idea.
 
Given a finite simplicial complex $K$, consider the class of simplicial complexes in which every point has a neighborhood that is simplicially isomorphic to $K$. This condition can be strengthened by fixing a vertex $v \in K$ and requiring that the local isomorphisms send each $0$-cell to $v$. The family of $3$-valent graphs arises in this way. For another example, consider those surfaces that can be triangulated with a fixed number $n$ of $2$-dimensional simplices around every vertex. For $n=5,6,7$, this construction yields spherical, Euclidean, and hyperbolic surfaces, respectively.

The pair $(K,v)$ as above is called a {\em model} centered on $v$. A result of Cooper and Thurston \cite{MR935525} shows that there is a set of 5 models, $\EuM$, so that every closed orientable 3-manifold admits a triangulation \emph{modeled} on $\EuM$. This means that every vertex of the triangulation admits a neighborhood isomorphic to one of the 5 models.
 
The goals of this paper are threefold. First, to describe what it means for a family of $n$-manifolds to be {\em locally combinatorially defined} (LCD) using the ideas above. Roughly, a family is LCD if every manifold in the family is modeled on a finite set $\EuM$ \emph{and} no other manifolds are modeled on $\EuM$. Second, to discuss {\em branched manifolds} and the family of manifolds that immerse into them. A family of compact piecewise linear (PL) $n$-manifolds is \emph{BM} if it consists of exactly those manifolds that immerse into a particular compact PL branched $n$-manifold. The third goal is to demonstrate the equivalence of these two ideas and introduce the notion of a {\em universal branched manifold} for such families. This is the content of the main theorem. Both directions are rather surprising.

\begin{theorem} \label{main thm}
A family of compact PL $n$-manifolds is LCD if and only if it is BM.
\end{theorem}

In future papers, \Cref{main thm} will be used to show that each of the eight Thurston geometries produce families of manifolds that are LCD. In particular, this means that there is a universal branched manifold for each of these geometries. \Cref{branched manifolds} provides some general results about PL branched manifolds that will be essential in these subsequent papers. Note that the definition of a PL branched manifold (see \Cref{def:PLbranched}) has been generalized from an earlier version of this paper.

{\Cref{example} gives an example of a universal branched manifold for those 3-manifolds that are torus bundles over the circle. \Cref{main thm} implies that this family of manifolds is LCD. The reader might prefer to start there. 

The main theorem is related to the theory of laminated spaces \cite{CLARK2014669},\cite{10.1063/1.1613041}.
These appear in many contexts, including dynamical systems and tiling spaces.
Given a finite set of models, the disjoint union of all triangulated manifolds
based on these models can be made into a compact metric space that has
a codimension-0 foliation with transverse set a Cantor set: a {\em matchbox manifold}.
The distance between two vertices is $2^{-n}$ if they have maximal isomorphic
 simplical neighborhoods of size $n$.
This metric space is an inverse limit of branched manifolds. We thank an anonymous 
reviewer for pointing out this connection.

 This paper, as well as the subsequent ones, is based on the results of the PhD thesis of the second named author \cite{lesliethesis}, supervised by the third author, with assistance from the first author. 
\gap

\noindent\textbf{Acknowledgments:} The authors thank Marc Lackenby, Mladen Bestvina, and Federico Ardila for helpful discussions. The third author was partially supported by National Science Foundation CAREER Grant DMS–2046889.

\section{Combinatorial Characterizations} \label{background}
General references for background on piecewise linear (PL) topology are \cite{MR248844},\cite{MR665919}. For those who use ZFC, in some places below, the word {\em set} should be replaced by {\em proper class}. We begin this section with some basics on PL and simplicial topology. 

 A \emph{ polyhedron} is the underlying space of a locally finite simplicial complex. If $P$ and $Q$ are polyhedra, then a map $f:P\rightarrow Q$ is PL  if there is a subdivision of the triangulation of $P$ such that the restriction of $f$ to each simplex in $P$ is an affine map with image contained in some simplex of the triangulation of $Q$.

A \emph{triangulated manifold} is a simplicial complex whose
underlying space is a manifold. A \emph{PL  structure on a manifold} is a maximal atlas such that the transition maps are PL. In particular, the definition of a PL manifold does not include a decomposition into simplices. Throughout, we assume that all manifolds are connected. 

If $K$ is a simplicial complex, \emph{ the $n$-skeleton of $K$} is the subcomplex $K^{(n)}$ that consists of all simplices of dimension at most $n$. In particular, $K^{(0)}$ is the set of vertices of $K$. The \emph{degree} of a vertex in a simplicial complex is the number of edges incident to it.

Suppose that $X$ is a connected simplicial complex. The \emph{simplicial metric} $d_{X}$ on $X^{(0)}$ is defined as follows. Given $u,v\in X^{(0)}$, then $d_X(u,v)$ is the smallest integer $k$ such that there is a sequence of vertices $$u=x_0,\ x_1,x_2,\cdots,x_{k-1},\ x_k=v\in X^{(0)},$$ where $x_i$ and $x_{i+1}$ are the endpoints of a $1$-simplex for each $i$.

The \emph{$r$-neighborhood of a vertex $v\in X$} is the subcomplex $N(X,v,r)$ consisting of the union of all simplices $\sigma$ in $X$ such that every vertex of $\sigma$ is simplicial distance at most $r$ from $v$. The \emph{star} of a vertex $v\in X^{(0)}$ is $\star(v):=\star(v, X)=N(X,v,1)$. 

A \emph{local model of dimension $n$} is a pair $(K,v)$ where $K$ is a simplicial complex, $v\in K^{(0)}$, and $|K|$ is PL homeomorphic to $[0,1]^n$. The vertex $v$ is called the \emph{center} of the model. Let $\EuM$ be a finite set of local models of dimension $n$. A triangulated $n$-manifold $M$ is \emph{modeled on $\EuM$} if and only if for every vertex $x\in M^{(0)}$ there exists a simplicial neighborhood $U\subset M$ of $x$, such that $(U,x)$ is simplicially isomorphic to some model $(K,v)\in \EuM$.

If $\EuM$ is a finite set of local models, then $\TR(\EuM)$ is the set of triangulated manifolds modeled on $\EuM$. The set of PL manifolds that are PL homeomorphic to an element of $\TR(\EuM)$ is denoted by $\PL(\EuM)$.

A set $\EuT$ of PL $n$-manifolds is \emph{locally combinatorially defined}, or \emph{LCD},  if it is $\PL(\EuM)$ for some finite set of local models $\EuM$. 

\section{Labeled Models}

In this section, we discuss adding \emph{labels} to simplices. The concept of a \emph{labeled local model} is then introduced using labeled simplices, where the set of labels is finite. The main result of this section (\Cref{labels are LCD}) shows that labels can be encoded into the combinatorics of triangulations, i.e. that \emph{labeled LCD} is equivalent to LCD. 

\begin{definition}
Suppose that $X$ is an $n$-dimensional simplicial complex and define $\lab(X)$ to be the set of vertices and $n$-simplices in $X$. A {\em labeling of $X$} is a function $\phi:\lab(X)\rightarrow L$. 

The set $L$ is called {\em the set of labels}.  If $v$ is a vertex of $X$, then $\phi(v)$ is called a {\em vertex label}. 
If $\sigma$ is an $n$-simplex of $X$ then $\phi(\sigma)$ is called a {\em simplex label}.

A {\em labeled simplicial complex} is a pair $(K,\phi)$ where $K$ is a simplicial complex and $\phi$ is a labeling of $K$. The complex $K$ is called the {\em underlying complex} of $(K,\phi)$. 
 \end{definition}
 
If $(K,\phi)$ and $(K',\phi')$ are labeled simplicial complexes, then a simplicial map $f:K\rightarrow K'$ is {\em label-preserving} if $\forall x\in \lab(K)$ then $\phi'(f(x))=\phi(x)$.

\begin{definition} 
A {\em labeled local model} is a triple $(K,v,\phi)$ where $(K,\phi)$ is a labeled simplicial complex and $(K,v)$ is a local model.
If $\EuM$ is a finite set of labeled local models, then a triangulated manifold $M$ is {\em modeled on $\EuM$} if there is labeling $\phi':\lab(M)\rightarrow L$ such that every vertex of $(M,\phi')$ has a neighborhood that is label-preserving simplicially isomorphic to a model in $\EuM$.
\end{definition}

If $\EuM$ is a finite set of labeled local models, then $\TR(\EuM)$ is the set of labeled triangulated manifolds modeled on $\EuM$. The set of PL manifolds that are PL homeomorphic to an element of $\TR(\EuM)$ is denoted by $\PL(\EuM)$. A set $\EuT$ of PL manifolds is {\em labeled LCD} or {\em LLCD} if $\EuT=\PL(\EuM)$ for some finite  set $\EuM$ of labeled local  models.

Labels can be encoded combinatorially using repeated subdivisions as follows.
 
 \begin{definition}
Suppose $x$ is a point in the interior of an $n$-simplex $\sigma$, where $n \geq 1$. Replacing $\sigma$ by the cone $x*\bdy\sigma$ is called {\em stellar subdivision}. It subdivides $\sigma$ into $(n+1)$ simplices $x*\tau$, one for each face $\tau$ of dimension $(n-1)$  of $\sigma$.
\end{definition}

\begin{lemma}(Standard Subdivision)\label{lem:subdivision}
 Suppose that $n\ge 2$ and $\sigma$ is an $n$-simplex
with vertices $v_0,\cdots,v_n$. Fix an integer $N>0$. There is a subdivision $K$ of $\sigma$ (see   \Cref{stellar subdivision 2}) called a {\em  standard subdivision} with all the new vertices in the interior of $\sigma$, so that $\bdy\sigma$ is not subdivided. Moreover, the vertices $v_i$ in $K$ have distinct degrees larger than $N$ and the degree of every vertex in the interior of $K$ is less than $2n+3$.
\end{lemma}

\begin{proof}
Suppose that $\tau$ is a face of dimension $(n-1)$ of an $n$-simplex $\sigma$. Consider a sequence of stellar subdivisions that produces a sequence of $n$-simplices $$\sigma=\sigma_0\supset\sigma_1\supset\cdots\supset\sigma_k,$$  where $\sigma_{i+1}=x_{i+1}*\tau$ is an $n$-simplex obtained from $\sigma_i$ by stellar subdivision using a point $x_{i+1}$ in the interior of $\sigma_i$. We call the resulting subdivision of $\sigma$ a {\em $k$-chain subdivision of $\sigma$,
based on $\tau$}.

Let $v_0,v_1,\cdots,v_n$ be the vertices of a labeled $n$-simplex $\sigma$. Suppose that $\tau_i$ is the $(n-1)$-dimensional face of $\sigma$ that contains all these vertices except $v_i$. Perform one stellar subdivision of $\sigma$ at $x_0$. Then, for each $0\le i\le n$, perform a $(N+i)$-chain subdivision based on $\tau_i$.
The degree of $v_j$ after this is
$$e(v_j):=n+1+\sum_{i\ne j} (N+i).$$

Therefore, the degrees of $v_j$ are all distinct and larger than $N$. Note that the degree of $x_0$  is $2(n+1)$ and the degrees of all other vertices added during the $k$-chain subdivisions are $n+1$ or $n+2$.
\end{proof}

\begin{figure}[h!]
\centering
    \begin{overpic}[scale = 0.35]{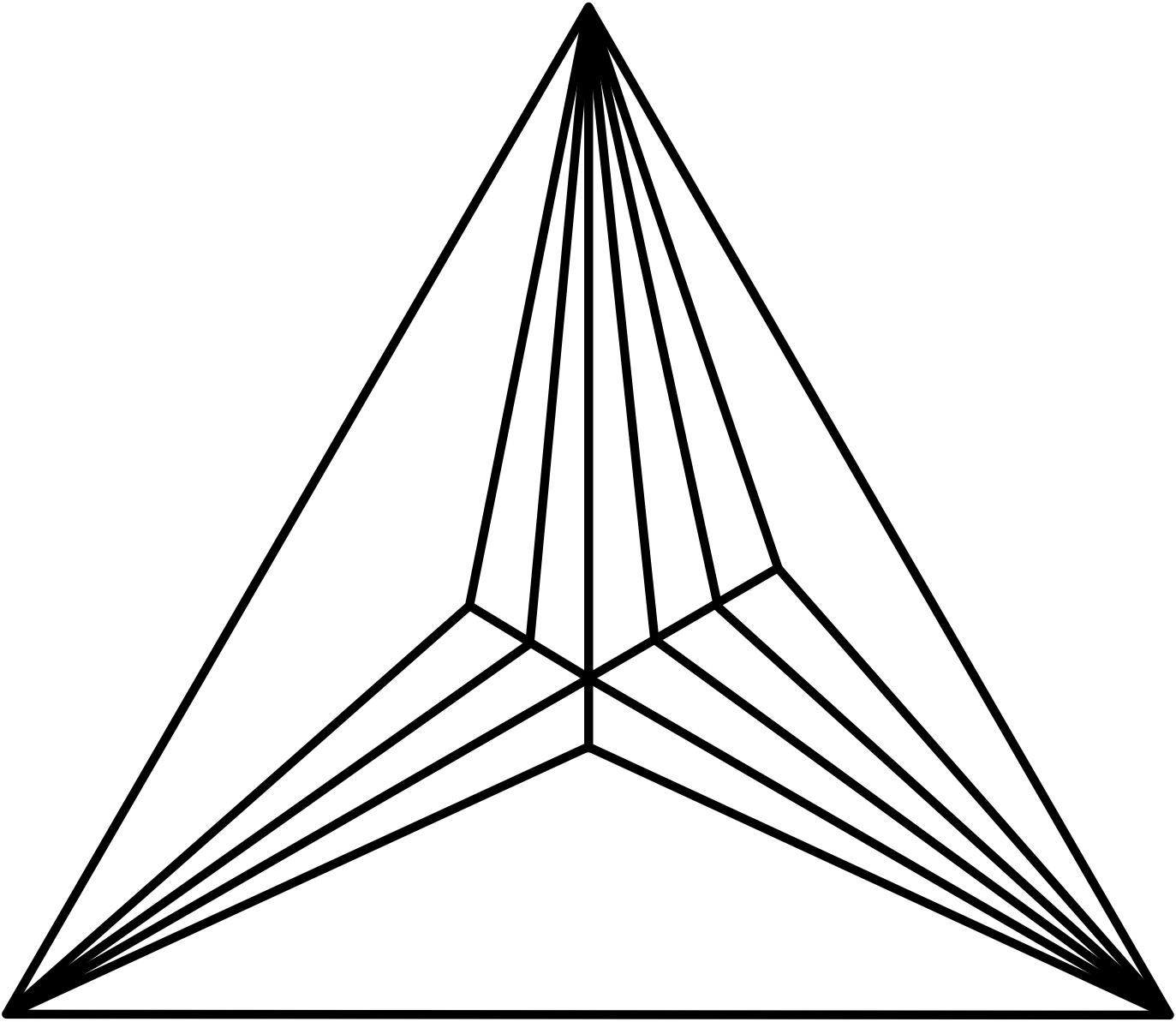}
    \end{overpic}
    \vspace{0.3 cm}
    \caption{A standard subdivision of a 2-simplex $\sigma$, where the degrees of the vertices in $\partial \sigma$ after subdivision are 6, 7, and 8.}
    \label{stellar subdivision 2}
\end{figure}

\begin{lemma}\label{labels are LCD}
    If a family of PL manifolds is LLCD, then it is LCD. 
\end{lemma}

\begin{proof} 
The idea of the proof is to use \Cref{lem:subdivision} to replace each $n$-simplex $\sigma$ in each labeled model by a standard subdivision $K(\sigma)$, where the triangulation on $K(\sigma)$ depends on the labels.

The new triangulation must have the property  that 
the original triangulation can be read off after subdivision. The standard subdivisions given by \Cref{lem:subdivision} add many new vertices in the interiors of the $n$-simplices that all  have {\em bounded degree}. We can choose the degrees in $K(\sigma)$ of the  vertices of $\sigma$ to be larger than this bound, which allows the original triangulation to be recovered.

Suppose that an LLCD family is given by  a finite set of labeled local models $\EuM$.  Let $Z$ be a  set consisting of one labeled $n$-simplex from each label-preserving isomorphism class of labeled $n$-simplices appearing in $\EuM$.

By \Cref{lem:subdivision}, there is a set $Y$ that consists of $|Z|$ distinct standard subdivisions of an $n$-simplex with the following properties. Each complex $K(\sigma)\in Y$ is a triangulation of an $n$-simplex $\sigma$ such that $\bdy K(\sigma)=\bdy\sigma$. Every vertex $v\in\bdy K(\sigma)$ has degree larger than $\Lambda=2(n+1)$, and every pair of distinct vertices in $\bdy K(\sigma)\cup \bdy K(\sigma')$ has different degrees for any $n$-simplices $\sigma$ and $\sigma'$. In particular, the standard subdivisions $K(\sigma)$ are {\em asymmetric} which means that there are no non-trivial simplicial automorphisms of the triangulation. 

Now choose a bijection $h:Z\rightarrow Y$ and for each $n$-simplex $\sigma \in Z$ choose a linear homeomorphism $H_\sigma: \sigma \rightarrow K(\sigma)$. This determines a degree label $e(v) > \Lambda$ for each vertex $v$ in $Z$. Let $\phi(v)$ be the label given by $\EuM$ to $v$.
Then, there is a well-defined map that sends $e(v)\mapsto\phi(v)$. In general, this map is not injective.
 
Define a set of unlabeled models $\EuM^*$ by taking the models in $\EuM$ and replacing each $n$-simplex label isomorphic to $\sigma$ by the standard subdivision $K(\sigma)$ using $H_\sigma$. 
For each model $(X,v,\phi) \in \EuM$, let $X^*$ be the resulting subdivision. Then $\EuM^*$  contains the model $(X^*, v)$, but it also contains models $(X^*, v')$ where $v'$ ranges over the set of  vertices in $X^*$ that are not in $X$. This ensures that the new vertices in $X^*$ are the center of some model in $\EuM^*$.

The original triangulated models and their vertex labels can be recovered from the unlabeled complexes using the map $e(v)\mapsto\phi(v)$.

We claim that $\PL(\EuM) = \PL(\EuM^*)$. The subdivision process immediately gives us that $\PL(\EuM) \subseteq \PL(\EuM^*)$.
  
For the other containment, suppose a manifold $M$ with triangulation $T^*$ is modeled on $\EuM^*$. Since the models in $\EuM^*$ are unions of standard subdivisions of $n$-simplices, $M$ is covered by standard subdivisions. The following claim shows that the decomposition of $M$ into standard subdivisions is unique. Suppose $A$ and $B$ are subcomplexes of $T^*$, each isomorphic to a standard subdivision. 
\gap 

\noindent \textbf{Claim:} If $A\neq B$ and $A\cap B\neq\emptyset$, then $A\cap B=\bdy A\cap\bdy B$, which is a single simplex in $M$. 
\gap

\noindent \textbf{Proof of claim:} The underlying spaces of $A$ and $B$ are $n$-balls in an $n$-manifold. In particular, $\bdy A$ separates the interior of $A$ from the interior of the complement. The vertices in $\bdy A$ and $\bdy B$ are the only vertices in $A\cup B$ with degree larger than $\Lambda$. Hence, no vertex of $\bdy A$ is in the interior of $B$ and no vertex of $\bdy B$ is in the interior of $A$. Thus, if $A\ne B$, using the fact $\bdy A$ and $\bdy B$ both separate, then the vertices in $A\cap B$ are all in $\bdy A\cap\bdy B$. 

Let $v$ be a vertex in $\partial A\cap \partial B$. There is a neighborhood $U$ of $v$ in $M$ that is simplicially isomorphic to a model in $\EuM^*$. There is a subset $A'$ of $U$ that is a standard subdivision of an $n$-simplex such that $A\cap A'$ contains points in the interior of $A$. The paragraph above implies that $A=A'$. Similarly, there is a
standard subdivision $B'\subset U$ and $B=B'$. Since $U$ is isomorphic to a model in $\EuM^*$, we have $A'\cap B' = A \cap B$ is one simplex.  \hfill $\blacksquare$\\

The claim implies that there is a unique decomposition of $M$ into standard subdivisions. Using the bijection between standard subdivisions in $Y$ and labeled simplices in $Z$, replace each standard subdivision $A$ in $M$ by one $n$-simplex $\sigma_A$. The result is a triangulated manifold $N$ that is PL homeomorphic to $M$. 

The simplex label on $\sigma_A$ is uniquely determined by the bijection. Moreover, for every vertex $u$ in $\sigma_A$, the vertex degree $e_A(u)$ of $u$ in $A$ determines its label. However, each vertex $u$ in $N$ is contained in several such $\sigma_A$. To show that the labeling on $N$ is well-defined, note that there is a neighborhood of $u\in M$ that is isomorphic to some $(X^*,v)\in\EuM^*$. It follows that $e_A(u)$ always encodes the vertex label of $v$ in $X$, independent of $A$. Thus, the labeling of $N$ is well-defined. In addition, $(X^*, v)$ is obtained by subdivision from a labeled local model $(X,v, \phi)\in \EuM$. Therefore, that $N$ is modeled on $\EuM$ and $M \in \PL(\EuM)$.
\end{proof}

\section{Colors and Geographies}\label{sec:colors}
Two kinds of vertex labels are important in this paper. The first is an assignment of a {\em color} in such a way that nearby vertices always have different colors.
The second is the assignment of a {\em geography}, which encodes the simplicial structure and color labels of vertices in a larger neighborhood of the vertex than is given by a local model. Since we only use the vertex labels, simplex labels can be ignored in the remainder of the paper.

Below, color and geography labels are added to a set of local models to produce a new set of local models, which has underlying simplicial complexes that are copies of the original models. Despite the added structure, the main result of this section (\Cref{refinement}) says that both sets of models define the same set of PL manifolds.

\begin{definition}
Suppose that $X$ is a connected simplicial complex and let $\phi:\lab(X)\rightarrow C$ be a labeling. The labeling is called a \emph{$d$-coloring} if, for all $v \in X^{(0)}$, the map $\phi$ is injective on the vertices of $N(X, v,d)$. The codomain $C$ is called the \emph{set of colors}.
\end{definition}

\begin{proposition}\label{finitecolors}  Given $d>0$ and a finite set of local models  $\EuM$  there is a finite set $C$ of colors such that every triangulated manifold modeled on $\EuM$ has a $d$-coloring using the colors $C$.
\end{proposition}

\begin{proof} 
Since $\EuM$ is finite, there is a universal bound on the number of vertices in $N(M, u, d)$ for any vertex $u$ in any manifold $M$ modeled on $\EuM$. The result follows by basic results about chromatic number in graph theory.
\end{proof}  
 
\begin{definition}
Suppose that $X$ is a simplicial complex.
A labeling $$\phi=(\phi_C,\phi_G):X^{(0)}\rightarrow C \times G$$ is called a {\em $d$-geography} if $\phi_C$ is a $d$-coloring of $X$ and $\phi_G$ has the property that whenever $u,v\in X^{(0)}$, then $\phi_G(u)=\phi_G(v)$ if and only if there is a simplicial isomorphism
$$f:N(X,u,d)\rightarrow N(X,v,d)$$ that preserves the color labels given by $\phi_C$. Such a map $\phi_G$ is called a \emph{geography labeling}.
\end{definition}

There is no requirement on the topology of $N(X,v,d)$, for example it might not be simply connected. The isomorphism $f$ is unique because no two vertices in $N(X,v,d)$ have the same color, so the image of each vertex  $w\in N(X,u,d)$ is the unique vertex in $N(X,v,d)$ with the same color label as $w$.

One way to obtain a $d$-geography is to define $G$ to
consist of one representative from each color-isomorphism class of colored $d$-neighborhoods in $X$, and let $\phi_G(v)$ be the element of $G$ color-isomorphic to $(N(X,v,d), \phi_C|)$.

\begin{definition}
A set of labeled models $\EuM$ is a \emph{$d$-geography  labeling} if for all manifolds $M$ modeled on $\EuM$, the labeling of $M$ is a $d$-geography.
\end{definition}

\begin{lemma}\label{refinement}
Given a finite set of models $\EuM$, there exists a $d$-geography labeling $\EuM_G$ such that $\PL(\EuM) = \PL(\EuM_G)$. 
\end{lemma}

\begin{proof}
Let $C$ be the set of colors given by \Cref{finitecolors}. Then every triangulated manifold modeled on $\EuM$ has a $d$-coloring using $C$. Let $X$ be the disjoint union of all $d$-colored manifolds modeled on $\EuM$ (using the colors $C$). 

Let $G$ be a set consisting of one representative from each equivalence class of colored subcomplexes $N(X, v, d)$. The equivalence relation is given by a simplicial isomorphism that preserves color labels. Notice that $G$ is finite due to the finiteness of $\EuM$ and $C$.

Let $Y$ be the set of all  pairs $(M,\phi)$, where $M$ is a triangulated manifold modeled on $\EuM$ and $\phi=(\phi_C,\phi_G)$. Here $\phi_C$ is a $d$-coloring of $M$ and the geography labeling $\phi_G$ of $M$ is defined by: $\phi_G(v)$ is the element of $G$ color-isomorphic to $(N(M,v,d), \phi_C|)$.

Next, we define the $d$-geography labeling $\EuM_G$. Given $(M,\phi)\in Y$ and $v$ a vertex of $M$, there exists a color and geography labeled simplicial neighborhood $U\subset M$ of $v$ such that $(U,v)$ is simplicially isomorphic to an element of $\EuM$. The set of labeled models $\EuM_G$ is defined by taking one representative in each equivalence class of such labeled neighborhoods $(U,v,\phi|_U)$. Here, $(U,v,\phi|_U)$ is equivalent to $(U',v',\phi'|_{U'})$ if and only if there is a simplicial isomorphism
$$f:(U,v) \rightarrow (U', v') \quad\textnormal{ and }\quad \phi=\phi'\circ f.$$
It is easy to check that $\TR(\EuM) = \TR(\EuM_G)$, which implies the lemma.  
\end{proof}

\section{Branched manifolds} \label{branched manifolds}
In what follows, the term {\em manifold} includes the case where the manifold has boundary.
Informally, a {\em branched $n$-manifold} is the result of gluing some $n$-cells together so that, at every point, there is a locally defined projection onto $[0,1]^n$ such that the restriction
to each cell is injective. Branched manifolds were introduced by Williams \cite{MR348794}. In dimension $1$, they are called {\em train tracks}, see for example \cite{ph16}. In dimension $2$, they are {\em branched surfaces}, see for example \cite{MR924769}. \Cref{local brnched man} shows the local structure of each. Our definition is inspired by, but different to, that of Williams. Not only do we work in the PL category instead of the smooth category, but also  his cocycle condition is replaced by a condition that charts are only {\em locally} compatible. Unless otherwise stated, $\Wcal$ denotes a PL branched $n$-manifold as an homage to Williams.

\begin{definition}\label{def:PLbranched} A \emph{ PL branched manifold of dimension $n$} is a pair $(\Wcal,\Pi)$ where $\Wcal$ is a polyhedron and $\Pi=\{ \pi_i:\EuW_i \rightarrow[0,1]^n\ |\ i\in I\}$ is an {\em atlas} that
is a collection of PL maps called \emph{ charts} that satisfy the  following conditions:  
\begin{enumerate}
\item[(i)] Every point in $\Wcal$ has a neighborhood that is some $\EuW_i$.
\item[(ii)] A \emph{ sheet in $\EuW_i$} is 
 $D\subset \EuW_i$ such that $\pi_i|_D:D\rightarrow[0,1]^n$  is a PL homeomorphism. The \emph{ sheet condition} is the requirement that $\EuW_i$ is the union of finitely many of its sheets.
 \item[(iii)] Two  charts $\pi_i$ and $\pi_j$ are {\em  compatible} if, for
  every point $x$ that has neighborhood in $\Wcal$ contained in $\EuW_i\cap\EuW_j$, there is a neighborhood $U\subset \EuW_i\cap\EuW_j$ of $x$ such that $$\forall\ p,q\in U\quad \pi_i(p)=\pi_i(q)\Leftrightarrow \pi_j(p)=\pi_j(q).$$
 The \emph{ compatibility condition} is the requirement all charts are compatible.
 \end{enumerate}

The map $(\pi_i|_D)^{-1}:[0,1]^n\rightarrow D$ is called a \emph{ parameterized sheet}. The \emph{ boundary} $\bdy\Wcal$ of $\Wcal$ is the union over $i\in I$ of the set of all $x\in \EuW_i$ such that $\EuW_i$ is a neighborhood of $x$ and $\pi_i(x)\in\bdy [0,1]^n$.  The \emph{manifold part} of $\Wcal$ is the subset of points which have a neighborhood that is contained in a sheet, and the complementary set is called the \emph{branch set}. 
 \end{definition}

 \begin{figure}[h!]
    \centering
    \includegraphics[scale = 0.4]{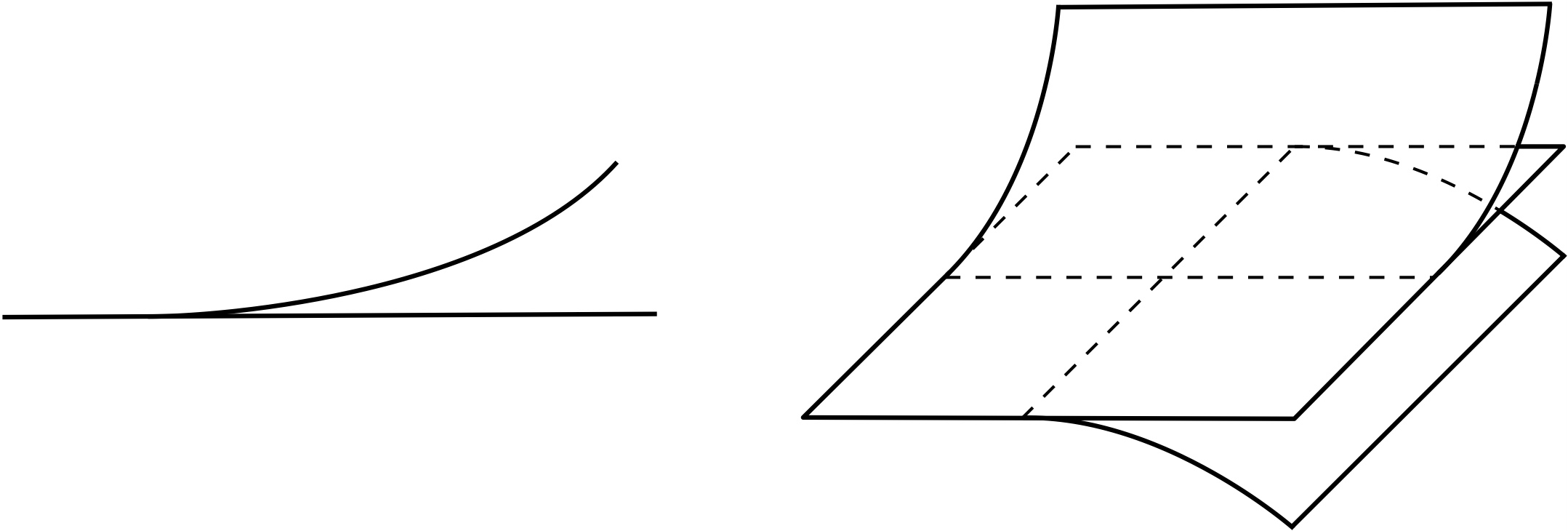}
    \caption{An example of the local structure of a branched 1-manifold and a branched 2-manifold.}
    \label{local brnched man}
    \vspace{-0.3 cm}
\end{figure}

 A \emph{chart} is a generalization of local coordinates on a manifold. A \emph{parameterized sheet} is a generalization of a parameterized neighborhood in a manifold. The branch set has codimension at least 1.  The {\em manifold part} is indeed a manifold and the compatibility requirement is automatic there. A PL branched manifold need not be connected.

The compatibility condition
has been weakened from an earlier version of this paper, since the new local condition is easier to use. In addition, a polyhedron, and hence a PL branched manifold, is no longer compact by definition. A compact branched manifold is a branched manifold where $\Wcal$ is a compact polyhedron. Every such branched manifold is isomorphic to one with a \emph{finite} atlas.

The compatibility condition ensures the existence of a generalization of tangent space for PL branched manifolds without boundary, by ensuring the equivalence relation in the following definition is independent of the choice of $\pi$.

\begin{definition}
    The \emph{tangent space} $T_x(\Wcal,\Pi)$ at a point $x\in\Wcal$ is the set of equivalence classes of PL maps $f:[0,1]\rightarrow \Wcal$ with $f(0) = x$ under the equivalence relation $f\sim g$ if there is $\epsilon>0$ such that
$\pi\circ f=\pi\circ g$ on $[0,\epsilon]$ where $\pi$ is a chart whose domain is a neighborhood of $x$. 
\end{definition}

A PL manifold is also a branched manifold, and it follows that $\pi$ identifies $T_x(\Wcal,\Pi)$ with 
$T_{y}[0,1]^n$, where $y = \pi(x)$. A PL curve has an initial segment that is linear, so this equivalence class is determined by a point in $\RR^n$ and $T_{y}[0,1]^n\cong\RR^n$.   A different choice of chart changes this identification by a PL homeomorphism. The definition above could be used to define a notion of branched immersion, however we will not use this idea. Instead, we employ  maps that locally send sheets into sheets.

\begin{definition}  Suppose that $(\Wcal,\Pi)$ and $(\Wcal',\Pi')$ are PL branched manifolds. A PL function
 $f:\Wcal\rightarrow\Wcal'$ is a {\em branched map} 
 if for every sheet $D$ in $\Wcal$ and every point $x$ in the relative interior of $D$, there is a neighborhood $U$ of $x$ in $D$ and a sheet $D'\subset \Wcal'$ such that $f(U)\subset D'$. Moreover, if $f$ is locally injective then $f$ is a {\em branched immersion}, and $f$ is a {\em branched homeomorphism} if $f$ is a homeomorphism and both $f$ and  $f^{-1}$ are branched maps. 
 \end{definition}

Note that if $M$ is a PL manifold then a PL map $\Omega:M\rightarrow\Wcal$ is an immersion exactly if for every chart $\pi\in\Pi$, the PL map $\pi\circ \Omega$ is locally injective wherever it is defined.

\begin{definition}  Suppose that $(\Wcal,\Pi)$ and $(\Wcal',\Pi')$ are PL branched manifolds. A PL function
 $f:\Wcal\rightarrow\Wcal'$ is a    \emph{proper map} if  $f^{-1}(\bdy\Wcal')=\bdy \Wcal$ and  the preimage of every compact set is compact.
 \end{definition}
 
All immersions in this paper are assumed to be proper. In order to prove that the family of PL manifolds that immerse into $(\Wcal, \Pi)$ is LCD, we first triangulate $\Wcal$. 

\begin{definition} A triangulation $T$ of a branched manifold $(\Wcal,\Pi)$ is called {\em nice}
if the branch set of $\Wcal$ is
a subcomplex of $T$ and for each simplex $\sigma$ there is a chart $\pi_i:\Wcal_i\rightarrow[0,1]^n$ such that $\Wcal_i$ is a neighborhood of $\sigma$.
\end{definition}

\noindent Since $\Wcal$ is a polyhedron, it is the underlying space of a simplicial complex and therefore admits a nice triangulation.

A proper immersion between compact manifolds of the same dimension is a covering map. Thus, a triangulation of the codomain pulls back to a triangulation of the domain. We show that the
same is true for a proper immersion of a compact manifold into a compact branched manifold of the same dimension when the branched manifold  has a {\em nice} triangulation.

\begin{lemma}\label{pullback} Suppose that $(\Wcal,\Pi)$ 
is a nicely triangulated compact PL branched $n$-manifold. Let $M$ be a compact PL $n$-manifold and $\Omega:M\rightarrow\Wcal$ be a proper PL immersion. There is a triangulation $T$ of $M$ called the {\em pullback} with the property that $\Omega:M_T\rightarrow \Wcal$  is simplicial.
\end{lemma}

\begin{proof}
Let $\sigma$ be an $n$-simplex of $\Wcal$. It suffices to prove that if $C$ is the closure
of a component of $\Omega^{-1}(\intr(\sigma))$, then $\Omega|_C:C\rightarrow\sigma$ is a homeomorphism.

 Since $\Wcal$ is the underlying space of an $n$-dimensional simplicial complex, $\intr(\sigma)$
is an open subset of $\Wcal$. There is a sequence of
subspaces $\sigma_i\subset\sigma$, each homeomorphic to $\sigma$, with $\sigma_i\subset\intr(\sigma_{i+1})$
 and $\bigcup\sigma_i=\intr(\sigma)$. Thus, for each $i$, the $n$-manifold $\intr(\sigma)$ is a neighborhood of $\sigma_i$ in $\Wcal$.
 
  Let $C_i$ be a component of $\Omega^{-1}(\sigma_i)$. Since $A=\Omega^{-1}(\intr(\sigma))$ is open, $A$ is an $n$-manifold in $M$ that is a neighborhood of $C_i$. Now, $\Wcal$ is nicely triangulated so there is a chart $\pi_j:\Wcal_j\rightarrow[0,1]^n$ such that $\Wcal_j$ is a neighborhood of $\sigma$. Since $\intr(\sigma)$ is disjoint from the branch locus, $\sigma$ is contained in a sheet, and thus $\pi_j|_\sigma$
is a homeomorphism onto its image. Since $\pi_j\circ\Omega|_A$ is a locally injective map between $n$-manifolds, it is a local homeomorphism. But $\pi_j|_\sigma$ is a homeomorphism, so
$\Omega|_A:A\rightarrow\intr(\sigma)$ is a local homeomorphism between manifolds.  It follows that each component $C_i$ of
 $\Omega^{-1}(\sigma_i)\subset M$ is a manifold with boundary.
 
 Since $\Omega$ is proper, $C_i$ is compact.
 Thus, $\Omega|_{C_i}:C_i\rightarrow\sigma_i$ is a proper immersion between compact connected manifolds of the same dimension. Therefore, $\Omega|_{C_i}$ is a covering map and $C_i$ is homeomorphic to $\sigma_i$. 
 If $U$ is a component of  $\Omega^{-1}(\intr(\sigma))$, then 
 $U=\bigcup\, C_i$ for suitable choices of components $C_i$. Hence,
 $\Omega|_U:U\rightarrow\intr(\sigma)$ is a  bijection. It is also continuous and is a local homeomorphism, so it is an open map. Thus, it is a homeomorphism.
 
 Let $V=\cl(U)$, then by continuity $\Omega(V)\subset\sigma$. Now, $V$ is compact so that $\Omega(V)$ is a compact subset of $\sigma$ that contains $\intr(\sigma)$. 
 Hence, $\Omega(V)=\sigma$. We claim that $\Omega|_V$ is injective. Suppose that
 $x_1,x_2\in V$ and $y=\Omega(x_1)=\Omega(x_2)$. Let $U_1,U_2$ be disjoint neighborhoods of $x_1$ and $x_2$. Then, there exists $z\in \Omega(U_1)\cap\Omega(U_2)\cap\intr(\sigma)$. But this means that there are distinct points $u_i\in U_i$ with $\Omega(u_1)=z=\Omega(u_2)$. Since $u_1,u_2\in U$, this contradicts that $\Omega|_U$ is injective. Thus, $\Omega|_V$ is a continuous injective map from the compact set $V$ to the Hausdorff space $\sigma$ so that $\Omega|_V$ is a homeomorphism.
\end{proof}

The remainder of this section develops some general results about branched manifolds that might be skipped on a first reading. These results will not be used in this paper, but will be used in future papers in the series. The main result here is the Gluing Theorem (\Cref{thm: covering brnch man}), which describes how to produce a branched manifold by gluing two branched $n$-manifolds together using a finite-degree covering map along boundary components. We begin by defining an equivalence relation on branched manifolds. 

 \begin{definition} The two branched manifolds $(\Wcal,\Pi)$ and $(\Wcal,\Pi')$ are
 {\em isomorphic} if $\Pi\cup\Pi'$ is a set of compatible charts.
     \end{definition}
     Given a branched manifold $(\Wcal,\Pi)$, let $\Pi'$ be the set of all charts that are compatible with $\Pi$. Then all the elements of $\Pi'$ are compatible with each other. Thus, there is a unique maximal atlas for each isomorphism type of branched manifold. This is analogous to the definition of a smooth manifold using a {\em maximal} atlas. If $\Wcal$ is compact, then $(\Wcal, \Pi)$ is isomorphic to
     a structure with a finite atlas.

If two branched manifolds are isomorphic, then the identity map is a branched homeomorphism. However, the converse is not necessarily true. For example,  let $\Wcal=\{(x,y)\in\RR^2: y(y-x^3)=0\}$  and $\pi,\pi':\Wcal\rightarrow\RR$ given by $\pi(x,y)=x$ and $\pi'(x,y)=x+y$.
 Then $\Pi=\{\pi\}$ and $\Pi'=\{\pi'\}$ are non-isomorphic branched manifold structures on $\Wcal$ even though the identity map is a branched homeomorphism.
 
 Similarly, two atlases on the same space may define branched manifolds that are not branched homeomorphic. For example, there are different atlases on the figure 8 graph that produce train tracks that are not branched homeomorphic.

   The following lemma shows that the isomorphism type of a branched manifold is local, in the sense that it
only depends on the branched structure in an arbitrarily small neighborhood of each point.

\begin{lemma}\label{equiv} Suppose that
$(\Wcal,\Pi)$ is a PL branched manifold, and $\EuU$ is an open cover of $\Wcal$. Then there is an isomorphic branched manifold $(\Wcal,\Phi)$ such that
the domain of every 
chart in $\Phi$ is contained in some element of $\EuU$.
\end{lemma}
\begin{proof} We may assume $\Pi$ is maximal. Define $\Phi\subset\Pi$ to be the subset of charts such that the domain is contained in some
element of $\EuU$. Then $\Phi$ is a compatible set of charts. Thus, it remains to show that $\Phi$ is an atlas for $\Wcal$. 

There is a chart $\pi: \EuW \to [0,1]^n$ in $\Pi$ and open set $U\in\EuU$ such that  $V=U\cap \operatorname{int}(\EuW)$ is a neighborhood of $x$.  
Thus, there is a small cube $C\subset \pi(V)\subset[0,1]^n$ that contains $\pi(x)$. Let $\EuV$ be the component of $\pi^{-1}(C)$ that contains $x$ and let
$f:C\rightarrow[0,1]^n$ be a linear homeomorphism. Define $\phi = f\circ \pi:\EuV\rightarrow[0,1]^n$. Then, $\phi$ is compatible with every element of $\Pi$ and satisfies the domain condition that implies it is in $\Phi$. Therefore, for every $x \in \Wcal$ there exists a chart in $\Phi$ whose domain is a neighborhood of $x$. Thus, $(\Wcal, \Phi)$ is a branched manifold that is isomorphic to $(\Wcal, \Pi)$ since $\Phi \subset \Pi$. 
\end{proof}

\begin{definition} Suppose $(\Wcal,\Pi)$ is a PL branched manifold,
 $\Wcal'$ is a subspace of $\Wcal$, and $(\Wcal',\Pi')$
 is a branched manifold structure on $\Wcal'$. Then $\pi \in \Pi$ and $\pi' \in \Pi'$ are {\em  compatible} if, for
  every point $x$ that has neighborhood in $\Wcal'$ contained in $\EuW\cap\EuW'$, there is a neighborhood $U\subset \EuW\cap\EuW'$ of $x$ such that $$\forall\ p,q\in U\quad \pi(p)=\pi(q)\Leftrightarrow \pi'(p)=\pi'(q).$$
  Then $(\Wcal', \Pi')$ is a \emph{branched submanifold} of $(\Wcal, \Pi)$ if $\Pi \cup \Pi'$ is a compatible set.
\end{definition}

Note that if $(\Wcal', \Pi')$ is a branched submanifold of $(\Wcal, \Pi)$ then the inclusion map $\Wcal'\hookrightarrow\Wcal$ is a branched immersion.

\begin{theorem}[Product Theorem]\label{thm:product}
    If $(\Wcal, \Pi)$ is a PL branched $n$-manifold and $(\Vcal, \Phi)$ is a PL branched $m$-manifold, let $\Psi = \{\psi=\pi\times\phi \ | \ \pi \in \Pi, \phi \in \Phi\}$.  The {\em branched product}  $(\Wcal \times \Vcal, \Psi)$ is a PL branched $(n+m)$-manifold.
    
   For any $v \in \Vcal$ and $w \in \Wcal$ there are natural inclusion maps $i_v: \Wcal \hookrightarrow \Wcal \times \Vcal$ and $i_w: \Vcal \hookrightarrow \Wcal \times \Vcal$ where $i_v(\Wcal)=\Wcal \times\{v\}$ and $i_w(\Vcal)= \Vcal \times \{w\}$. Moreover, there are natural projection maps $p_1: \Wcal \times \Vcal \to \Wcal$
and $p_2: \Wcal \times \Vcal \to \Vcal$. 
All these are  branched maps.
\end{theorem}

\begin{proof}
 The sheets of $\Wcal \times \Vcal$ are of the form $D \times E$ where $D$ is a sheet for $\Wcal$ and $E$ is a sheet for $\Vcal$. Thus, the sheet condition holds for $(\Wcal \times \Vcal, \Psi)$. This fact can also be used to easily check that the inclusion and projection maps are branched maps. The compatibility condition on $(\Wcal \times \Vcal, \Psi)$ follows from the compatibility condition on $(\Wcal, \Pi)$ and $(\Vcal, \Phi)$. 
\end{proof} 

A {\em Riemannian metric} on a branched manifold is a metric whose restriction to each sheet is Riemannian.
A \emph{Klein geometry} is a pair $(G,X)$ where $G$ is a Lie group that acts transitively on the manifold $X$ by real analytic maps. A \emph{ geometric structure} on a branched manifold  is an identification of  each sheet with a subspace of $X$ in such a way that transition maps are in $G$. The notion of developing map and holonomy then extend to this setting.

\begin{definition}\label{def: (G,X)} A \emph{$(G,X)$ structure} on a PL branched manifold $(\Wcal, \Pi)$ is a family of local homeomorphisms $\{h_{\pi}:[0,1]^n\rightarrow X\ |\ \pi\in \Pi \}$  such that every transition map $(h_{\pi'}\circ\pi')\circ(h_{\pi}\circ\pi)^{-1}$ is the restriction of an element of $G$.\
\end{definition}

We now turn to the Gluing Theorem which is an essential component in constructing universal branched manifolds for the eight Thurston geometries. 
\begin{remark}\label{rem: homeo glue}
For $i = 1,2$, let $\Wcal_i$ be a branched manifold, $A_i\subset\bdy\Wcal_i$ be a boundary
component that contains no branched set, and $p:A_1\rightarrow A_2$ be
a homeomorphism.  Then it is easy to define a branched manifold structure
on $\Wcal=(\Wcal_1\sqcup\Wcal_2)/p$ so that $\Wcal_i$ is a branched submanifold; the argument follows almost exactly the argument for gluing two manifolds homeomorphically along boundary components to obtain another manifold.
\end{remark}
\Cref{thm: covering brnch man} generalizes this construction to the case when $p$ is a covering map of finite degree, and further to the HNN case where one boundary component of a branched manifold is mapped to another by a covering map. The proof relies on the following covering space construction on collar neighborhoods.

Suppose that $M$ is a PL manifold and $\rho: \widetilde{M} \to M$ is a finite degree covering map. Let $N=M\times[-1,1]$ and $\widetilde{N}=\widetilde{M}\times[-1,1]$. Then $p:\widetilde{N} \to N$ defined by $p = (\rho, \operatorname{id})$ is a cover of finite degree. Define $\Vcal(N,p)=\Vcal=\widetilde{N}/\sim$ where
$$(x,t)\sim(x',t')\quad\Leftrightarrow\quad (t\ge0\ \ \&\ \ p(x,t)=p(x',t')).$$
   Thus, $\Vcal$ is constructed by gluing $\widetilde{M}\times[-1,0]$ to $M\times[0,1]$ by the covering map $\rho$ on $\widetilde{M} \times 0$. Let $r:\widetilde{N}\rightarrow\Vcal$ be the quotient map and let $q:\Vcal\rightarrow N$ be the map defined by $q([x,t])=p(x,t)$.
   Then $q$ is also a quotient map and $p=q\circ r$.

Given a manifold structure $(M,\Pi)$,
suppose $\pi\in\Pi$ and $U=domain(\pi)$. 
Define a chart $\pi'$ on $\Vcal$ with domain 
$\EuV=q^{-1}(U\times[-1,1])$ by $\pi' = \pi\circ(q|_{\EuV})$.
If $\widetilde U$ is a lift of $U$ to $\widetilde M$ then $r(\widetilde U\times[-1,1])$ is a sheet in $\EuV$. Thus, $\EuV$ is a union of finitely many sheets.
The well-definedness of $q$ implies that the collection of all such charts are compatible and defines a branched manifold
structure on $\Vcal=\Vcal(N,p)$.

\begin{theorem}[Gluing Theorem]\label{thm: covering brnch man}
Let $(\Ucal, \Phi)$ be a PL branched $n$-manifold, possibly not connected. Suppose that $A$ and $B$ are two different components of $\bd\Ucal$ that are disjoint from the branch set and $p: A \to B$ is a PL covering map  of finite degree. Let $\Wcal=\Ucal/\sim$ where $x\sim p(x)$ if $x\in A$ with corresponding quotient map $q:\Ucal\rightarrow\Wcal$.

 Then, there exists a PL branched $n$-manifold $(\Wcal, \Pi)$ with the following properties: there is a neighborhood  $C\subset \Wcal$ of $q(B)$ that
 is branched homeomorphic to $\Vcal(A,p)$ and the complement is branched homeomorphic to the complement of a collar neighborhood of $(A\cup B)$ in $\Ucal$.
\end{theorem}

\begin{proof}
Let $N$ be a collar neighborhood of $(A \cup B)$ in $\Ucal$. By an application of \Cref{equiv},  $\Ucal' = \overline{\Ucal - N}$ can be equipped with a branched manifold structure so that it is a branched submanifold of $\Ucal$. Now, $\Wcal$ can be viewed as the disjoint union of $\Vcal = \Vcal(A,p)$ and $\Ucal'$ glued by a homeomorphism $\bd \Ucal' \to \bd  \Vcal$. By \Cref{rem: homeo glue}, $\Wcal$ is the desired branched $n$-manifold. 
\end{proof}

\section{Equivalence Theorem and Universal Branched Manifolds}\label{sec: equivalence}

\begin{definition} $\PL(\Wcal,\Pi)$ is the family of compact PL  manifolds that properly PL immerse into $(\Wcal,\Pi)$, and $(\Wcal,\Pi)$ is called a \emph{universal branched manifold} for this family. \end{definition}

A closed manifold $M$ is a special case of a branched manifold, and in this case, $\PL(M)$ is the set of finite coverings of $M$.

The main result of the paper is proven in this section. The first half of \Cref{main thm} is: 

\begin{theorem}\label{brnch implies LCD}
Suppose $(\Wcal, \Pi)$ is a compact PL branched $n$-manifold. Then, $\PL(\Wcal, \Pi)$ is LCD. 
\end{theorem} 

\begin{proof}
Let $(\Wcal, \Pi)$ be nicely triangulated. By \Cref{labels are LCD} it suffices to show that there exists a finite set of {\em labeled} models $\EuM$ such that $\PL(\Wcal, \Pi) = \PL(\EuM)$.

For every $M \in \PL(\Wcal, \Pi)$, by \Cref{pullback}, there is a pullback triangulation $T$ so that $\Omega: M_T \to \Wcal$ is simplicial. The map $\zeta_M = \Omega|_{M^{(0)}}$ gives a labeling of $M$ with labels in $\Wcal^{(0)}$. Let $$\EuScript{X} := \{ (M, \Omega) \ | \ M \textnormal{ is a PL manifold and } \Omega: M \to \Wcal \textnormal { is an immersion }\},$$

\noindent and let  
$$\EuScript{Y} := \{ (\star(u),u, \zeta_M| ) \ | \ (M, \Omega) \in \EuScript{X}, \ u \in M_T^{(0)}\}/\sim,$$ 
where $(\star(u), u, \zeta_M|) \sim (\star(v), v, \zeta_N|)$ iff there exists a label-preserving simplicial isomorphism between them. The set of local models $\EuM$ is defined by taking one representative of every equivalence class in $\EuScript{Y}$. 

If $u$ is a vertex in $M_T$, then by the definition of a simplicial immersion, $\star(u)$ is label-preserving simplicially isomorphic to a subcomplex of $\Wcal$. Since $\Wcal$ is compact, the number of such subcomplexes is finite. Therefore, $\EuM$ is finite. By definition, if $M \in \PL(\Wcal, \Pi)$, then $M_T$ is label-modeled on $\EuM$. 

On the other hand, suppose $N$ is a triangulated manifold label-modeled on $\EuM$. Define $\Psi: N \to \Wcal$ where a $k$-simplex $\sigma$ in the triangulation $T$ of $N$ is mapped linearly and label-preservingly to the unique $k$-simplex in $\Wcal$ whose vertex labels match those of $\sigma$.

We claim that $\Psi$ is an immersion. It suffices to show that $(\star(x),x)$ maps into a sheet of $\Wcal$.  Since $N$ is modeled on $\EuM$, the neighborhood $(\star(x),x)$ is label-preserving simplicially isomorphic to $(\star(u),u) \subseteq M$ where $(M, \Omega) \in \EuScript{X}$. The result follows from the fact that $\Psi(\star(x), x) = \Omega(\star(u), u)$.
\end{proof}

The next goal is to construct a branched manifold from a set of local models $\EuM$ that will be a universal branched manifold for the LCD family $\PL(\EuM)$. In general, universal branched manifolds are not unique.

\begin{theorem} \label{immersion into W}
Suppose that $\EuM$ is a finite set of local models. Then, there is a compact branched manifold $(\Wcal,\Pi)$ such that $\PL(\Wcal,\Pi)=\PL(\EuM)$.
\end{theorem}

\begin{proof}
Let $\EuM$ be a finite set of local models. Let $d\ge 2$ be larger than the maximum diameter of a local model in $\EuM$. Use \Cref{refinement} to construct a $d$-geography labeling $\EuM_G$ from $\EuM$. Let $M$ be any labeled manifold that is modeled on $\EuM_G$ and let $\phi=(\phi_C, \phi_G)$ be the $d$-geography labeling of $M$. If $v$ is a vertex in $M$  then $\phi_C(v)$ is the color of $v$ and $\phi_G(v)$ is a colored complex such that $$\phi_G(v)\cong_c (\ N(M,v,d)\ ,\ \phi_C|\ )$$ where $\cong_c$ means simplicially isomorphic preserving color labels. 

The proof of this theorem involves many simplicial complexes with color and geography labels. To keep track of them all, we introduce the following two simplices. Let $C$ be the finite set of colors and $G$ the finite set of geographies. Define the {\em color simplex} $\Delta_C$ to be the simplex with vertex set $C$ and the {\em geography simplex} $\Delta_G$ to be the simplex with vertex set $G$. These two simplicial complexes have natural color and geography labels, respectively. In each simplex, there is a unique vertex with a given label. We refer to the labels on $\Delta_G$ as {\em pseudo-geography labels} because they determine a neighborhood of a vertex in some manifold modeled on $\EuM_G$, {\bf not} a neighborhood of a vertex in $\Delta_G$.

Let $(Q,\phi_C,\phi_G)$ be the disjoint union of all the labeled triangulated manifolds in $\TR(\EuM_G)$. 
 There are two canonical simplicial maps defined on the vertices by
  $$(\Theta_C,\Theta_G):Q\rightarrow\Delta_C\times \Delta_G\qquad
  (\Theta_C,\Theta_G)(v)=(\phi_C(v),\phi_G(v)).$$
  The maps $\Theta_C, \Theta_G$ are color and geography preserving, respectively. If a subcomplex $X$ of $Q$ has vertices with distinct color labels, then $\Theta_C|: X \to \Delta_C$ is an embedding. The same is true for $\Theta_G|_X$ if $X$ has distinct geography labels.

There is also a canonical projection $\Psi:\Delta_G\rightarrow\Delta_C$ that sends a geography vertex $(N(M, v, d), \phi_C|)$ in $\Delta_G$ to the vertex in $\Delta_C$ that has the same color as $v$. 
  Define the subcomplex $\Wcal\subset\Delta_G$  by $$\Wcal=\Theta_G(Q).$$ 
  That is to say, $\Wcal$ is obtained by identifying simplices in manifolds modeled on $\EuM_G$ when they have the same geography labels.

For each $x\in \Wcal^{(0)}$ define a subcomplex $\Wcal_x=\star(x,\Wcal)$. 
By definition of $\Wcal$, $$\star(x,\Wcal)=\bigcup_{q\in\Theta_G^{-1}(x)} \Theta_G(\star(q,Q)).$$
If $\Theta_G(q)=x=\Theta_G(q')$ then,
 by definition of the geography label, $$\Theta_C(\star(q',Q))=\Theta_C(\star(q,Q)),$$ so 
$\Psi(\Theta_G(\star(q,Q)))$ is independent of the choice of $q$ in $\Theta_G^{-1}(x)$.
Therefore, $$\Psi(\star(x,\Wcal))=\Theta_C(\star(q,Q)).$$
Since $\star(q,Q)$ is colored distinctly, $\star(q,Q)$ is color-isomorphic to $\Theta_C(\star(q,Q))$. Therefore, there is a PL homeomorphism $h_x:\Psi(\star(x,\Wcal))\rightarrow [0,1]^n$. Define the chart $$\pi_x:\Wcal_x\rightarrow[0,1]^n\qquad{\rm by}\quad  \pi_x=h_x\circ\Psi|_{\Wcal_x}.$$
Let $\Pi = \{\pi_x: \Wcal_x \to [0,1]^n \, | \, \, x \in \Wcal^{(0)}\}$. Since $\Wcal$ is compact, the set $\Pi$ is finite.

\gap
\noindent \textbf{Claim: }$(\Wcal,\Pi)$ is a branched manifold.
\gap
 
\noindent \textbf{Proof of Claim:} By definition of $\pi_x$,
we have that $\Theta_G(\star(q,Q))$ is a sheet in $\Wcal_x$ for all $q \in \Theta_G^{-1}(x)$, so that $\Wcal_x$ is the union of sheets. This is the {\em sheet condition}.

Take $p, p' \in \Wcal_x \cap \Wcal_y$. The charts $\pi_x$ and $\pi_y$ are both the composition of $\Psi$ with a PL homeomorphism. Therefore, $\pi_x(p)= \pi_x(p')$ if and only if $\Psi(p) = \Psi(p')$, which is if and only if $\pi_y(p) = \pi_y(p')$. Thus, the {\em compatibility condition} holds.
This completes the proof of the claim. \hfill$\blacksquare$ 

\gap

To prove the theorem, it remains to show that $\PL(\Wcal,\Pi)=\PL(\EuM)$. By \Cref{refinement} we have $\PL(\EuM)=\PL(\EuM_G)$, so it suffices to prove that $$\PL(\Wcal, \Pi) = \PL(\EuM_G).$$

Since every manifold modeled on $\EuM_G$ is $d$-colored, every model of $\EuM_G$ contained in a $d$-neighborhood of such a manifold maps injectively by $\Theta_G$ into $\Wcal$. Together with the compatibility condition, this implies that $\Theta_G|_M$ is an immersion for any manifold $M$ modeled on $\EuM_G$. Therefore,  $\PL(\Wcal,\Pi)\supseteq\PL(\EuM_G)$.

The reverse containment is more involved. Begin by defining $\EuM^*$ to be a set of pseudo-geography labeled sheets in $\Wcal$. Observe that no two sheets are label-isomorphic. We will show that $$\PL(\Wcal, \Pi) \subseteq \PL(\EuM^*) \subseteq \PL(\EuM_G).$$ 

If a PL manifold $M$ admits an immersion $\Omega$ into $\Wcal$, then \Cref{pullback} gives a triangulation of $M$ such that $\Omega: M \to \Wcal$ is simplicial. This implies that the star of every vertex $v\in M$ is label-isomorphic to a sheet in $\Wcal$. Thus, $M$ is modeled on the set $\EuM^*$ so that $\PL(\Wcal, \Pi) \subseteq \PL(\EuM^*)$.

 To prove the containment $\PL(\EuM^*)\subseteq\PL(\EuM_G)$, suppose that $M\in\TR(\EuM^*)$ contains a vertex $v$. The manifold $M$ has color and geography labels, which we also strategically call $\phi = (\phi_C, \phi_G)$. Let $\Theta_G(v) = x \in \Wcal$. Fixing a point $q$ in $Q$ with $\Theta_G(q) = x$, we have $\phi_G(v) = \phi_G(q)$, and the geography $\phi_G(q) = (N(Q, q, d), \phi_C|)$ contains a model in $\EuM_G$ that is centered on $q$. The theorem then follows from:
 
\gap
\noindent {\bf Claim:}  $N(M,v,d)\cong_c \phi_G(v).$
\gap

\noindent\textbf{Proof of Claim:} In the argument below, all complexes have diameter at most $d$. Because all relevant manifolds are $d$-colored, the complexes are embedded in $\Delta_C$ by $\Theta_C$. By identifying a complex $X$ with its image $\Theta_C(X)\subset\Delta_C$, then color-preserving maps between complexes become inclusion maps in $\Delta_C$. The converse is also true; if $\Theta_C(X) \supset \Theta_C(Y)$ then $X$ contains a subcomplex that is color-isomorphic to $Y$. We will often omit $\Theta_C$ to avoid cumbersome notation.

Since $M$ is modeled on $\EuM^*$,  $\star(v,M)$ is label-isomorphic to a sheet in $\Wcal_x$. By the definition of $\pi_x$, the sheet is color-isomorphic to $\star(q,Q)$. Suppressing $\Theta_C$, we therefore have 
$$\star(v)=N(M,v,1)\subset\phi_G(v).$$

We refer to $\star(q)$ as the {\em center} of the geography $\phi_G(v)$. Thus, at each vertex $v$ in $M$, the center of the geography label of $v$ can be identified (via $\Theta_C$) with the star neighborhood of $v$ in $M$. 

Suppose that $e\subset M$ is an edge with endpoints $v$ and $u$. There is a model
neighborhood of $v$ in $M$ given by some element of $\EuM^*$, and $e$ is in this neighborhood.
 There is an edge $\tilde e\subset Q$ such that $\Theta_G(\tilde e)=\Theta_G(e)$ and the endpoints $q$ and $z$ of $\tilde e$ are identified with $v$ and $u$, respectively. This means that $\phi_G(u) = \phi_G(z)$. By definition of $\EuM^*$, it follows that
 $$\phi_G(v)=N(Q,q,d)\qquad \textnormal{ and } \qquad \phi_G(u)=N(Q,z,d).$$
Since $d(q,z)=1$, in $\Delta_C$
$$\phi_G(v)\supset N(\phi_G(u),z,d-1).$$
This containment says that the geographies at adjacent vertices of $M$ have large overlap. 
Suppose  $v=v_0,v_1,\cdots,v_{r}$ is  a sequence of vertices in $M$ such that $v_i$ is adjacent to $v_{i+1}$. Then, $$\phi_G(v)\supset N(\phi_G(v_r),q_r,d-r),$$
where $q_r \in Q$ is chosen as follows. Take the edge $e_r$ in $M$ with endpoints $v_{r-1}$ and $v_r$. There is an edge $\tilde e_r \subset Q$ such that $\Theta_G(\tilde e_r) = \Theta_G(e_r)$ and the endpoints of $\tilde e_r$ are identified with $v_{r-1}$ and $v_r$. Let $q_r$ be the end point that is identified with $v_r$ via $\Theta_G$. 

If $d-r\ge 1$, then the center, $\star(v_r)$, of each of the geographies $\phi_G(v_r)$ is contained in $\phi_G(v)$. For every vertex $p$ in $N(M,v,d)$, there exists a sequence $v, v_1, \cdots, v_r$ such that $p$ is in $\star(v_r)$ and $r \leq d-1$. Since $N(M, v, d)$ is the union of these star neighborhoods, in $\Delta_C$,
$$\phi_G(v)\supset N(M,v,d).$$
Given that $\star(v) \cong_c \star(q)$ and $\phi_G(v) = N(Q,q,d)$, the two sets are equal in $\Delta_C$.  \hfill $\blacksquare$

\end{proof}

\noindent \Cref{brnch implies LCD} and \Cref{immersion into W} together imply \Cref{main thm}.

\section{Examples}\label{example}
As a first example, we define a compact branched 3-manifold $(\Wcal,\Pi)$ such that $\PL(\Wcal,\Pi)$ consists of all torus bundles over the circle. The monodromy of every such bundle is an element of $\GL(2,\ZZ)$. Every element in $\GL(2,\ZZ)$ can be written as a finite product of the following matrices: 
 $$a_1=\bpmat1 & 1\\0 & 1\epmat\qquad a_2=\bpmat1& 0\\ 0& -1\epmat \qquad a_3=\bpmat0 & 1\\1 & 0\epmat$$
For $1\le i\le 3$, let $p_i:A_i\rightarrow S^1$ be a torus bundle over the circle with 
 monodromy $a_i$. Choose a point  $x\in S^1$  and set
 $B_i=p_i^{-1}(x)\subset A_i$.
 Identifying each $B_i$ with a fixed $T^2$, we have that $A_i\cap A_j=B_i=B_j$ for all $i,j$. Therefore, 
 $\Wcal=\bigcup_i \, A_i$ is connected. There is an obvious branched manifold structure $\Pi$ on $\Wcal$ such that the inclusion maps $A_i\hookrightarrow\Wcal$ are immersions.
 
Let $\Gamma$ be the branched 1-manifold (train track) that is the wedge of three copies of $S^1$ identified along $x$.
 Gluing the maps $p_1,p_2,p_3$ gives a map $q:\Wcal\rightarrow \Gamma$ such that $q|_{A_i}=p_i$ and $q(A_i)$ is the $i^{th}$ copy of $S^1$ in $\Gamma$. Thus, $\Wcal$ is a torus bundle over $\Gamma$ in a way that extends the three original bundles.
  
 An immersion $f:S^1\rightarrow\Gamma$ traverses a sequence of edges in $\Gamma$. This sequence determines a word (up to circular conjugacy) in the alphabet $\{a_1,a_2, a_3\}$ according to the sequence of edges. This word gives a matrix $c\in\GL(2,\ZZ)$, and every element of $\GL(2,\ZZ)$ is obtained in this way.
 Let $C$ be the torus bundle over $S^1$ with monodromy $c$.
 There is an immersion $F:C\rightarrow \Wcal$ that covers $f$. Thus, every torus bundle over $S^1$ immerses into $\Wcal$.
 
 If a closed $3$-manifold $M$ immerses into $\Wcal$, then the preimage of a torus fiber in $\Wcal$ is a union of tori in $M$. Hence, $M$ is a torus bundle over $S^1$.
 
 It follows from \Cref{main thm} that there is a finite set of local models $\EuM$ such that $\PL(\EuM)$
 consists of all 3-manifolds that are torus bundles over the circle. It is not too hard to construct such a set directly using LCD $\Leftrightarrow$ LLCD (\Cref{labels are LCD}).

 As another example, consider the set of all closed, orientable 3-manifolds. By \cite{MR935525}, there is a set of models $\EuM$ such that every closed, orientable 3-manifold is modeled on $\EuM$. However, one also obtains non-orientable manifolds with these models. To correct this, one can add an orientation label to each 3-simplex, and then use that LLCD implies LCD. Thus, there is a universal branched 3-manifold that all closed orientable 3-manifolds immerse into, and non-orientable ones do not.
   
A different proof of these results is obtained from \Cref{main thm} by constructing a universal branched manifold for closed, orientable 3-manifolds using the fact that every such 3-manifold has a Heegaard splitting. 
 
The fact that there is a universal branched manifold for closed hyperbolic 3-manifolds is more surprising and more difficult to prove.

\bibliography{legorefs.bib} 
\bibliographystyle{abbrv} 

\end{document}